\documentclass[12pt]{article}
\usepackage{amsfonts}
\usepackage{amsfonts}
\usepackage{mathrsfs}
\usepackage{graphicx}
\usepackage{caption2}
\usepackage{amsfonts}
\usepackage[dvipsnames,usenames]{color}
\usepackage{amsmath, amsthm, amsfonts, amssymb}
\usepackage[toc,page]{appendix}

\newtheorem{thm}{Theorem}[section]

\newtheorem{lem}[thm]{Lemma}

\newtheorem{defn}[thm]{Definition}

\theoremstyle{definition}
\numberwithin{equation}{section}
\theoremstyle{remark} \hsize=7.5truein \vsize=8.6truein

\def\R{{\hbox{\bf R}}}
\def\C{{\hbox{\bf C}}}

\def\P{{\hbox{\bf P}}}
\def\E{{\hbox{\bf E}}}

\def\be{{\bf{e}}}
\font \roman = cmr10 at 10 true pt

\def\ts{{T_\sigma}}

\def\be#1{ \begin{equation}\label{#1} }
\def\bas{\begin{align*}}
\def\eas{\end{align*}}
\def\bi{\begin{itemize}}
\def\ei{\end{itemize}}

\def\supp{{\hbox{\roman supp}}}
\def\re{{\hbox{\roman \bf Re}}}
\def\im{{\hbox{\roman\bf Im}}}

\def\Z{{\hbox{\bf Z}}}

\def\to{\rightarrow}

\def\emph#1{{\it #1}}
\def\textbf#1{{\bf #1}}

\def\Bu{{\mathbf u}}

\def\ts{{\tilde s}}

\def\tnu{{\tilde \nu}}
\def\txi{{\tilde \xi}}

\def\lam{{\lambda}}
\def\Var{{\hbox{\roman\bf Var}}}

\theoremstyle{plain}

\theoremstyle{remark}
  \newtheorem{remark}[subsection]{\bf Remark}

\theoremstyle{definition}

\begin{document}
\title{Local law for eigenvalues of random Hermitian matrices with external source }

\author{Linh Tran \\
Department of Mathematics, University of Washington}
\date{}
\maketitle

\begin{abstract}
We prove a local law for eigenvalues of the random Hermitian matrices with external source $W_n=\frac{1}{n}X_n+A_n$ where $X_n$ is Wigner matrix and $A_n$ is diagonal matrix with only two values $a, -a$ on the diagonal. The local law is an essential step to prove the universality conjecture for this random matrix model.
\end{abstract}

\section{Introduction}
Investigating the convergence of eigenvalue distribution  of Hermitian random matrix models is one of the clasic questions in random matrix theory. A typical example is Wigner's semi circle law for Wigner Hermitian matrices.In addition to the question about global convergence of the spetral distribution i.e. convergence on the whole support, there is an important question concerning "local law" i.e. convergence on intervals of very short length. These local laws allow one to establish other critical results like delocalization of eigenvector and they serve as fundamental steps in the proof of the universality conjecture. This conjecture states that the behavior of a random matrix model depends more on the algebraic structure of the matrix rather than the distribution of its entries. If one replace the entry's distribution by another with the same moments then the behavior more or less stays the same. There are many recent break through results on this universality conjecture by L. Erdos, B. Schlein, H-T Yau, T. Tao, V. Vu and others \cite{ESY_Local, EYYbulk, ESY_sine, TVUniBulk, TVUniEdge}.\\

In this short paper we prove a local convergence law for the spectral distribution of the random Hermitian matrix with external source model. The model is obtain by adding a deterministic pertubation matrix to the classical Wigner Hermitian matrix. This is an attempt toward proving the universality conjecture for this model.

\begin{defn}[Wigner matrix] A Wigner Hermitian matrix of size $n$ is a random Hermitian matrix $X_n$ which satisfies 
\begin{itemize}
\item The upper triangular complex entries $\zeta_{ij}=\xi_{ij}+\sqrt{-1}\tau_{ij}\ (1\le i<j\le n)$ where $\xi_{ij}$, $\tau_{ij}$ are iid copies of a real random varible $\xi$ following a probability measure $\nu$ with mean zero and variance $1/2$.
\item The diagonal real entries $\xi_{ii}\ (1\le i\le n)$ are iid copies of a real random variable $\txi$ following a probability measure $\tnu$ with mean zero and variance $1$.
\item $\xi,\ \txi$ have exponetial decay, i. e. there are constant $C,\ C'$ such that $\P(|\xi|\ge t^C)\le e^{-t}$ for all $t\ge C'$ (the same holds for $\txi$)
\end{itemize}
\end{defn}
The condition on exponential decay allows one to apply a standard truncation argument (see e.g. \cite{SBbook}) to get 
\begin{equation}\label{trunc.bnd} \sup_{1\le i,j\le n}|\zeta_{ij}|\log^{C+1}n 
\end{equation}
almost surely. Therefore we will assume (\ref{trunc.bnd}) for all of our proofs.\\

The {\it random matrix with external source} is defined as
 $$W_n=\frac{1}{\sqrt{n}}X_n+A_n,$$ 
where $X_n$ is a Wigner Hermitian matrix and $A_n$ is a deterministic matrix. Here we consider $A_n$ to be a diagonal matrix with only two values $a$ and $-a$ on the diagonal with the same multiplicity (so $n$ is even). Let $\lambda_1\le\lambda_2\dots\le\lambda_n$ be the real eigenvalues of $Wn$, the empirical spectral distribution (ESD) $\mu_n$ of $W_n$ is defined as
$$\mu_n=\frac{1}{n}\sum_{i=1}\delta_{\lambda_i},$$
where $\delta_\lambda$ is the Dirac point mass measure at point $\lambda$.\\

 In \cite{BK04} Bleher and Kuijlaars investigated convergence of the ESD of $W_n$ for the case $a>1$. 
\begin{thm}[Limiting ESD \cite{BK04}]\label{limESD} The measure $\mu_n$ converges in distribution to a limiting measure $\mu$ as $n$ goes to infinity. $\mu$ has a density function $\rho(x)$ which can be expressed as
$$\rho(x)=\frac{1}{\pi}|\im s(x)|,$$
where $s=s(x)$ solves the cubic Pastur's equation
$$ s^3-xs^2-(a^2-1)s+xa^2=0.$$
The density $\rho$ is real analytic on its support $(-z_1,-z_2)\cup(z_2,z_1)$ where $z_1, z_2$ are real numbers depending only on $a$.
\end{thm}

During the paper we will assume the Wigner matrix $X_n$ satisfies the following condition on its atom distributions\\

{\bf Condition ${\bf C_0}$:}{\it
\begin{itemize}
\item The probability measure $\nu$ and $\tnu$ satisfy the Poincare inequality, which means there exists a constant $C$ so that for any function $u$
\begin{equation}\label{poincare}
\int|u-\int ud\nu|^2d\nu\le C\int|\nabla u|^2d\nu,
\end{equation}
and the same hold for $\tnu$.
\item  The probability measure $\nu$ and $\tnu$ satisfy the log Sobolev inequality, which means there exists a constant $C$ so that for any density function $u$ with $\int ud\nu=1$
\begin{equation}\label{logSob}
\int u\log ud\nu\le C\int|\nabla \sqrt{u}|^2d\nu,
\end{equation}
and the same hold for $\tnu$.
\end{itemize}
}
 
Our main result is
 \begin{thm}[Local law for ESD]\label{thmlocallaw} Let $W_n$  be a random Hermitian matrix with external source satisfying condition ${\bf C_0}$ and  $|\zeta_{ij}|\le K$ almost surely for all $i,\ j$. For any $\epsilon, \delta >0$ and any interval $I\subset\R$ of width $|I|\ge\frac{K^2\log^{20}n}{n}$, the number of eigenvalues $N_I$ of $W_n$ in $I$ satisfies the concentration inequality
$$|N_I-n\int_I\rho(x)dx|\le\delta n|I|$$
with overwhelming probability.
\end{thm}

\begin{remark} It was our intention to use Theorem \ref{thmlocallaw} to prove the universality conjecture for the random Hermitian matrix with external source. The goal was an analogue of the Four Moment Theorem \cite{TVUniBulk} as the local law was the biggest challenge step, the rest is more or less the same as in the proof for Wigner matrix. However, very recently, S. O'Rouke and V. Vu \cite{OV} has announced a proof of the Four Moment Theorem for a Hermitian  random matrix with a more general external source matrix. So we will only present the proof of the local law which is slightly simpler and suitable for our setting.
\end{remark}

The rest of the paper is organized as follows: In section 2 we prove the fluctuation of the ESD which will serve as a critical step in the proof of Theorem \ref{thmlocallaw}. In Section 3 we present the proof of Theorem \ref{thmlocallaw}.\\
We will often use the phrase "the event $E$ holds with overwhelming probability" which means $\P(E)\ge 1- \exp(-\omega(\log n))$.

\section{Fluctuation of empirical statistic distribution}

\begin{lem}\label{lem.crubd} Let $W_n$  be a random Hermitian matrix with external source satisfying condition ${\bf C_0}$. Then for $C$ large enough,
$$N_I\le Cn|I|$$
with overwhelming probability.
\end{lem}
\begin{proof}
The proof of this lemma is exactly the same as that of Lemma 66 in \cite{TVUniBulk}. The only difference is the diagonal entries $\zeta_{ii}=\xi_{ii}+A_n(i,i)$, but we only need to consider the imaginary part of thesel entries, and $A_n(i,i)$ are real and don't contribute to the imaginary part. 
\end{proof}
Recall that the Stieltjes transform of a probability measure $\nu$ is defined as
$$s_{\nu}(z)=\int_\R\frac{d\nu(x)}{x-z}$$
for $z\in\C^{+}:=\{z:\im z>0\}$. Here we will denote $s(z)$ to be the Stieltjes transform of the limiting measure $\mu$, and $s_n(z)$ be the Stieltjes transform of the ESD of $W_n$, so
$$s_n(z)=\frac{1}{n}\sum_{k=1}^n\frac{1}{\lam_k-z}.$$
\begin{lem}[Fluctuation of ESD]\label{lem.ESDfluc} Let $W_n$  be a random Hermitian matrix with external source satisfying condition ${\bf C_0}$ and the entries $|\zeta_{ij}|\le K$ almost surely for all $i,\ j$. Fix $1\ge\eta\ge\frac{K^2\log^{19} n}{n}$, $x\in \supp(\mu)$ and let $z=x+\sqrt{-1}\eta$. Then\\
(i) There exists a constant $C$ so that
\begin{equation}\label{Stiel.var.ineq}
\Var(s_n(z))\le\frac{C}{n^2\eta^3}
\end{equation}
(ii) For any $\epsilon>0$ there exists a constant $c$ so that
\begin{equation}\label{Stiel.fluc.ineq}
\P(|s_n(z)-\E s_n(z)|\ge \epsilon)\le \exp(-cn\eta\epsilon\min\{(\log n)^{-1},n\eta^2\epsilon\})
\end{equation}

\end{lem}

\begin{proof} The proof mostly follows that of Theorem 3.1 of \cite{ESY_Local}. First we prove (i). Let $\Bu_1,\dots, \Bu_n$ be the orthonormal basis of eigenvectors of $W_n$ corresponding to the eigenvalues $\lambda_1,\dots,\lam_n$. By first order pertubation theory
\begin{align}
\frac{\partial \lam_k}{\partial \re \zeta_{ij}}&=\bar{\Bu}_k(i)\Bu_k(j)+\bar{\Bu}_k(j)\Bu_k(i)=2\re (\bar{\Bu}(i)\Bu(j))\\
\frac{\partial \lam_k}{\partial \im \zeta_{ij}}&=\sqrt{-1}[\bar{\Bu}_k(i)\Bu_k(j)-\bar{\Bu}_k(j)\Bu_k(i)]=2\im (\bar{\Bu}(j)\Bu(i))
\end{align}
for $1\le i<j\le n$ and 
$$\frac{\partial \lam_k}{\partial \zeta_{ii}}=\bar{\Bu}_k(i)\Bu_k(i)$$
for $i=1,\dots,n$ and $\zeta_{ii}=\xi_{ii}+A_n(i,i)$, and by the Poincare inequality we obtain
\begin{equation}
\begin{split}
\E|s_n(z)-\E s_n(z)|^2\le\ & C\sum_{i<j}\E\left(\left|\frac{\partial s_n(z)}{\partial \sqrt{n}\re \zeta_{ij}}\right|^2+\left|\frac{\partial s_n(z)}{\partial \sqrt{n}\im \zeta_{ij}}\right|^2\right)\\
&+C\sum_{i=1}^n\E\left|\frac{\partial s_n(z)}{\partial \sqrt{n}\re \zeta_{ii}}\right|^2\\
=&\frac{C}{n^3}\sum_{i<j}\E\left(\left|\sum_k\frac{1}{(\lam_k-z)^2}\frac{\partial\lam_k}{\partial\re\zeta_{ij}}\right|^2 + \left|\sum_k\frac{1}{(\lam_k-z)^2}\frac{\partial\lam_k}{\partial\im\zeta_{ij}}\right|^2 \right)\\
&+\frac{C}{n^3}\sum_{i=1}^n\E\left|\sum_k\frac{1}{(\lam_k-z)^2}\frac{\partial\lam_k}{\partial\zeta_{ii}}\right|^2\\
=&\frac{C}{n^3}\E\sum_{i<j}\sum_{k,l}\bigg(\frac{\re(\bar{\Bu}_k(i)\Bu_k(j))\re(\bar{\Bu}_l(i)\Bu_l(j))}{(\lam_k-z)^2(\lam_l-\bar{z})^2}\\
&\qquad\;\;\;\;\;\;\;\;\;\;\;\;\;\;+\frac{\im(\bar{\Bu}_k(j)\Bu_k(i))\im(\bar{\Bu}_k(j)\Bu_k(i))}{(\lam_k-z)^2(\lam_l-\bar{z})^2}\bigg)\\
&+\frac{C}{n^3}\E\sum_{i=1}^n\sum_{k,l}\frac{|\Bu_k(i)|^2|\Bu_l(i)|^2}{(\lam_k-z)^2(\lam_l-\bar{z})^2}\\
=&\frac{C}{n^3}\E\sum_{k,l}\frac{1}{(\lam_k-z)^2(\lam_l-\bar{z})^2}\sum_{i,j}\Bu_k(i)\overline{\Bu_l(i)}\Bu_l(j)\overline{\Bu_k(j)}\\
=&\frac{C}{n^3}\E\sum_{k}\frac{1}{(\lam_k-z)^4}
\end{split}
\end{equation}
For $h\in \Z$ define the intervals $I_h=[x+(h-\frac{1}{2})\eta,x+(h+\frac{1}{2})\eta]$. Take $K_0$ large enough so that $\sup_k|\lam_k|\le K_0$ with ovewhelming probability and $J$ be an interval so that  $[-K_0,K_0]\subset \cup_{h\in J}I_h$. Then
\begin{equation}
\begin{split}
\sum_{k}\frac{1}{|\lam_k-z|^4}&\le\sum_{h\in J}\sum_{k:\lam_k\in I_h}\frac{1}{|\lam_k-z|^4}+\sum_{k:|\lam_k|\ge K_0}\frac{1}{\eta^4}\\
&\le \frac{C}{\eta^4}\sup_{h\in J} N_{I_n}+\frac{1}{\eta^4}|\{k:|\lam_k|\ge K_0\}|
\end{split}
\end{equation}
Combine with Lemma \ref{lem.crubd} one get
\begin{equation}
\begin{split}
\Var s_n(z) &\le \frac{C}{n^3\eta^4}\E\sup_{h\in J}N_{I_n}\\
&\le \frac{C}{n^2\eta^3}
\end{split}
\end{equation}

Now we prove (ii). We will control the real part of $s_n(z)-\E s_n(z)$, the imaginary part is the same. Let $d\P$ denote the probability measure of $W_n$ and define a density function 
$$u=\frac{\exp(e^y\re(s_n(z)-\E s_n(z)))}{\int \exp(e^y\re(s_n(z)-\E s_n(z))) d\P}.$$
Then by the log-Sobolev inequality we obtain
\begin{equation}
\begin{split}
\frac{d}{dy}\Big[e^{-y}\log&\int \exp(e^y\re(s_n(z)-\E s_n(z))) d\P\Big]= e^{-y}\int u\log u d\P\\
\le&Ce^{-y}\int|\nabla\sqrt{u}|^2 d\P\\
\le&Ce^{-y}\int\left[\sum_{i<j}\left(\left|\frac{\partial s_n(z)}{\partial\sqrt{n}\re \zeta_{ij}}\right|^2+\left|\frac{\partial s_n(z)}{\partial\sqrt{n}\im \zeta_{ij}}\right|^2 \right)+\sum_{i=1}^n\left|\frac{\partial s_n(z)}{\partial\sqrt{n} \zeta_{ii}}\right|^2 \right] ud\P\\
\le&\frac{Ce^y}{n^3\eta^4}\left[\int\sup_{h\in J}N_{I_h} ud\P +\int|\{k:|\lam_k|\ge K_0\}|ud\P\right]\\
\le&\frac{Ce^y}{n^2\eta^3}
\end{split}
\end{equation}
 for $K_0$ large enough,  using similar arguments as in the proof of (i). Integrate this from $y=y_0$ to $y=\log L$ for some $L<\frac{n\eta}{C_1\log n}$, we get
$$\log\E \exp(L\re[s_n(z)-\E s_n(z)])\le Le^{-y_0}\log\E\exp(e^{y_0}\re[s_n(z)-\E s_n(z)])+\frac{CL^2}{n^2\eta^3},$$
where the first term of the r.h.s will vanish when $y_0\to\-\infty$. Thus
$$\E \exp(L\re[s_n(z)-\E s_n(z)])\le \exp(CL^2n^{-2}\eta^{-3}),$$
and
\begin{equation}
\begin{split}
\P(\re[s_n(z)-\E s_n(z)]\ge \epsilon)&\le \exp(CL^2n^{-2}\eta^{-3}-\epsilon L)\\
&\le\exp(-cn\eta\epsilon\min\{\log^{-1}n, n\eta^2\epsilon\})
\end{split}
\end{equation}
Repeat the argument for $-s_n(z)$ and we get the proof for (ii).
\end{proof}
\section{Proof of local law}
One important ingredient of the proof is the following Lemma proved by Capitaine et al. \cite{Capitaine11}
\begin{lem}[Convergence rate of ESD, Proposition 4.1 \cite{Capitaine11}]\label{lem.ESDConvRate} Let $W_n=\frac{1}{\sqrt{n}}X_n+A_n$ be a random Hermitian matrix with external source satisfying condition $C_0$. Let $\mu_n$ be the free convolution between the semicircle distribution and the spectral distribution of $A_n$, and $\ts_n(z)$ is the Stieltjes transform of $\mu_n$. Recall that $s_n(z)$ is the Stieltjes transform of the ESD of $W_n$. Then for a fixed $z\in\C^{+}$
$$|\E s_n(z)-\ts_n(z)| =O(1/n).$$
\end{lem}
In our problem the ESD of $A_n$ are the same $\mu_A=\frac{1}{2}(\delta_a+\delta_{-a})$ for every even $n$, so $\ts_n(z)=s(z)$ for every even $n$, where $s(z)$ is the Stieltjes tranform of $\mu=\mu_{sc}\boxplus\mu_A$. Combine Lemma \ref{lem.ESDConvRate} with part (ii) of Lemma \ref{lem.ESDfluc} we obtain that 
$$|s_n(z)-s(z)|\le o(1)$$
for all $z$ with $\re z\in \supp\mu$ and $\im z\ge\frac{K^2\log^{19}n}{n}$ with uniformly overwhelming probability. Theorem \ref{thmlocallaw} is then a direct consequence of the following Lemma
\begin{lem}[Control of Stieltjes transform implies control on ESD]\label{lem.ConStie} Let $1\ge \eta\ge1/n$ and $L,\epsilon,\delta>0$. Suppose one has the bound
$$|s_n(z)-s(z)|\le\delta$$
with uniformly overwhelming probability for all $z$ with $|\re z|<L$ and $\im(z)\ge \eta$. Then for any interval $I\subset[-L+\epsilon,L-\epsilon]$ with $|I|\ge\max(2\eta,\frac{\eta}{\delta}\log\frac{1}{\delta})$, one has 
$$|N_I-n\mu(I)|\ll_\epsilon \delta n|I|$$
with overwhelming probability.
\end{lem}
The proof of  Lemma \ref{lem.ConStie} is the same as that of Lemma 64 \cite{TVUniBulk}, with the semicirle density $\rho_{sc}(x)$ replaced by the density $\rho(x)$ of our limiting measure $\mu$.
\endproof


\bibliographystyle{plain}
\bibliography{random-matrices}

\end{document}